\documentclass{mathincs}

\usepackage{graphicx}
\usepackage{subcaption}
\usepackage{amsmath}
\usepackage{amssymb}
\usepackage{arydshln}
\usepackage{xfrac}

 \newtheorem{theorem}{Theorem}[section]
 
 \newtheorem{lemma}[theorem]{Lemma}
 \newtheorem{proposition}[theorem]{Proposition}
 \theoremstyle{definition}
 \newtheorem{definition}[theorem]{Definition}
 \theoremstyle{remark}
 \newtheorem{remark}[theorem]{Remark}
 \newtheorem{conjecture}[theorem]{Conjecture}
 
 \numberwithin{equation}{section}

\begin{document}
\renewcommand{\arraystretch}{1.15} 
%
%
%
\title[Minimal height companion matrices for Euclid polynomials]
 {Minimal height companion matrices \\for Euclid polynomials}
\author{Eunice Y.~S.~Chan}

\address{%
Ontario Research Center for Computer Algebra and School of Mathematical and Statistical Sciences, Western University}

\email{echan295@uwo.ca}

\author{Robert M.~Corless}
\address{Ontario Research Center for Computer Algebra and School of Mathematical and Statistical Sciences, Western University}
\email{rcorless@uwo.ca}


\date{\today}
\dedicatory{Dedicated to Jonathan M.~Borwein}

\begin{abstract}
We define Euclid polynomials $E_{k+1}(\lambda) = E_{k}(\lambda)\left(E_{k}(\lambda) - 1\right) + 1$ and $E_{1}(\lambda) = \lambda + 1$ in analogy to Euclid numbers $e_k = E_{k}(1)$. We show how to construct companion matrices $\mathbb{E}_k$, so $E_k(\lambda) = \operatorname{det}\left(\lambda\mathbf{I} - \mathbb{E}_{k}\right)$, of height 1 (and thus of minimal height over all integer companion matrices for $E_{k}(\lambda)$). We prove various properties of these objects, and give experimental confirmation of some unproved properties.
\end{abstract}

\maketitle
\section{Introduction}
The sequence $e_n = 2, 3, 7, 43, 1807, \ldots$ defined by $e_1 = 2$ and the recurrence relation
\begin{equation}
	e_{n+1} = e_{n}e_{n-1} \cdots e_2e_1 + 1 
\end{equation}
for $n \geq 1$, is known under various names: Euclid numbers, Sylvester's sequence, or Ahmes numbers. The sequence can be found at The Online Encyclopedia of Integer Sequences as entry A000058. There, we find references to work of Erd{\"o}s, Shparlinsky, Vardi, Sloane, Guy, and other well-known number theorists and analysts.

These numbers, which we will call Euclid numbers, as they are called in \cite[chapter 4]{grahamconcrete}, have interesting properties. For instance, they are mutually relatively prime. Quoting \cite{grahamconcrete},
\begin{quote}
	``Euclid's algorithm (what else?) tells us this in three short steps, because $e_n\operatorname{mod}e_{m} = 1$ when $n > m$: $\operatorname{gcd}(e_n, e_m) = \operatorname{gcd}(1, e_m) = \operatorname{gcd}(1, 0) = 1$."
\end{quote}
Euclid numbers grow \textsl{doubly exponentially}; indeed exercise 37, chapter 4 of \cite{grahamconcrete} asks the reader to prove\footnote{The hint there is to write $e_{n+1}-\sfrac{1}{2} = \left(e_n - \sfrac{1}{2}\right)^{2} + \sfrac{1}{4}$ and consider $2^{-n}\log\left(e_{n} - \sfrac{1}{2}\right)$.} that
\begin{equation}
	e_n = \left\lfloor E^{2^{n}} + \dfrac{1}{2}\right\rfloor
\end{equation}
for a number $E \approx 1.264$; here $\lfloor x \rfloor$ is the floor of $x$, the largest integer not greater than $x$.

The name ``Ahmes numbers" comes from a connection to so-called \textsl{Egyptian fractions}\footnote{Quoting Exercise 9, p.~95 from \cite{grahamconcrete}, ``Egyptian mathematicians in 1800 BC represented rational numbers between 0 and 1 as sums of unit fractions $\sfrac{1}{x_{1}} + \cdots + \sfrac{1}{x_k}$ where the $x_k$ were distinct positive integers."}. Quoting N{\'e}stor Romeral Andr{\'e}s from the A000058 entry,
\begin{quote}
	``The greedy Egyptian representation of 1 is $1 = \sfrac{1}{2} + \sfrac{1}{3} + \sfrac{1}{7} + \sfrac{1}{43} + \sfrac{1}{1807} + \cdots$"
\end{quote}
and he then goes on to give a geometric dissection of a unit square (in words) proving this assertion. Algebraically, we have the following.
\begin{lemma}\label{lemma:egypt}
\begin{equation}
	1 = \sum_{k = 1}^{n} \dfrac{1}{e_k} + \dfrac{1}{e_{n+1} - 1}
\end{equation}
because
\begin{align}
	e_{n + 1} &= e_{n}e_{n-1}\cdots e_1 + 1 \nonumber\\
	&=e_{n}\left(e_{n} - 1\right) + 1 \>. 
\end{align}
\end{lemma}
\begin{proof}
An easy induction: clearly $1 = \sfrac{1}{2} + \sfrac{1}{2} = \sfrac{1}{2} + \sfrac{1}{(3-1)}$ so the statement is true for $n = 1$. Then
\begin{align}
	1 &= \sum_{k=1}^{n} \dfrac{1}{e_k} + \dfrac{1}{e_{n+1} - 1} \nonumber \\
	&= \sum_{k=1}^{n+1} \dfrac{1}{e_k} + \dfrac{1}{e_{n+1} - 1} - \dfrac{1}{e_{n+1}} \nonumber \\
	&= \sum_{k=1}^{n+1} \dfrac{1}{e_k} + \dfrac{e_{n+1} - e_{n+1} + 1}{e_{n+1}\left(e_{n+1} - 1\right)} \nonumber \\
	&= \sum_{k=1}^{n+1} \dfrac{1}{e_k} + \dfrac{1}{e_{n+2} - 1} \>.
\end{align}
\end{proof}

There are other properties too, but we hope that this is enough to whet your appetite because we want to move on to what we call\footnote{The polynomials $E_{k}(-\lambda)$ occur, not with this name, as sequence A225200 by Martin Renner.} ``Euclid polynomials." Put
\begin{equation}
	E_{1}(\lambda) = \lambda + 1
\end{equation}
and
\begin{equation}
	E_{n+1}(\lambda) = \lambda E_{n}(\lambda) E_{n-1}(\lambda)\cdots E_{1}(\lambda) + 1
\end{equation}
for $n \geq 1$. Then, obviously, $E_{k}(0) = 1$ for $k \geq 1$ and $E_{k}(1) = e_k$ for $k \geq 1$. Possibly these polynomials in the variable $\lambda$ can shed some light on Euclid numbers. One could make $E_0(\lambda) = 1$ but this complicates later formulae to no purpose. The first few Euclid polynomials are
\begin{align}
	E_{1} &= \lambda + 1 \nonumber \\
	E_{2} &= \lambda^2 + \lambda + 1 \nonumber \\
	E_{3} &= \lambda^4 + 2\lambda^3 + 2\lambda^2 + \lambda + 1 \nonumber \\
	E_{4} &= \lambda^8 + 4\lambda^7 + 8\lambda^6 + 10\lambda^5 + 9\lambda^4 + 6\lambda^3 + 3\lambda^2 + \lambda + 1 \>.
\end{align}
We will enumerate and prove some properties of these polynomials in the next section, but first we confess: we're not interested in Euclid polynomials because of their connection to Euclid numbers. We are interested because we have a new technique for finding their roots, namely by finding an equivalent eigenvalue problem (a so-called ``companion matrix") that has a vary interesting property of its own, namely that out of all integer matrices $\mathbf{A}_k$ having
\begin{equation}
	E_k(\lambda) = \operatorname{det}\left(\lambda\mathbf{I} - \mathbf{A}_k\right)
\end{equation}
the height of $\mathbf{A}_k$---that is, the absolute value of the largest entry of $\mathbf{A}_k$---is the \textsl{least} when we use our method.
\begin{remark}
	$\operatorname{Height}(\mathbf{A}) = \|\operatorname{vec}(\mathbf{A})\|_{\infty}$ is actually a matrix norm. It is not, however, submultiplicative:
	\begin{equation}
		\mathbb{H}(\mathbf{AB}) \not\leq \mathbb{H}(\mathbf{A})\mathbb{H}(\mathbf{B}) \>.
	\end{equation}
	For example, consider
	\begin{equation}
		\left[
			\begin{array}{cc}
				2 & 2 \\
				2 & 2
			\end{array}
		\right] 
		=
		\left[
			\begin{array}{cc}
				1 & 1 \\
				1 & 1
			\end{array}
		\right]
		\left[
			\begin{array}{cc}
				1 & 1 \\
				1 & 1
			\end{array}
		\right]
		\>.
	\end{equation}
\end{remark}

We will find companion matrices for $E_k(\lambda)$ of height 1, as small as possible for any integer matrix. This is to be contrasted with the size of the largest polynomial coefficient of $E_k(\lambda)$, which since
\begin{equation}
	E_k(1) = \sum_{j=0}^{2^{k-1}}E_{j, k} = \left\lfloor E^{2^{k}} + \dfrac{1}{2}\right\rfloor
\end{equation}
must at least be
\begin{equation}
	\dfrac{1}{2^{k-1} + 1}\left\lfloor E^{2^{k}} + \dfrac{1}{2}\right\rfloor = \mathcal{O}\left(E^{2^{k} - \mathcal{O}(k)}\right)
\end{equation}
(the maximum cannot be smaller than the average). Here, we are denoting the coefficients of
\begin{equation}
	E_k(\lambda) = \sum_{j=0}^{\operatorname{deg}E_k} E_{j, k} \lambda^{j}
\end{equation}
by $E_{j, k}$ and claiming $\operatorname{deg}E_{k}(\lambda) = 2^{k - 1}$, which we will prove in the next section. This massive reduction in height has important numerical consequences. The eigenvalues of this ``minimal height companion matrix" will be much easier to compute than are the roots of the explicit polynomial (with its doubly-exponentially large coefficients).

This minimal height companion matrix would itself just be a curiosity, except that the technique we use to generate it turns out to be quite general, and in fact can be extended to \textsl{matrix} polynomials, giving so-called \textsl{lower-height linearizations}\footnote{\textsl{Minimal} height linearizations are an open question.}. Euclid polynomials have a special place in our hearts, though, because it was by finding their minimal height companion matrices that we realized the technique was, in fact, general.

\section{Properties of Euclid Polynomials}
\begin{proposition}
	$\operatorname{deg}E_{k}(\lambda) = 2^{k-1}$.
\end{proposition}
\begin{proof}
	$\operatorname{deg}E_{1}(\lambda) = \operatorname{deg}\lambda + 1 = 1 = 2^{1 - 1}$. Since
	\begin{align}
		E_{k+1}(\lambda) &= \lambda E_{k}(\lambda)E_{k-1}(\lambda) \cdots E_{1}(\lambda) + 1 \nonumber \\
		&= E_{k}(\lambda)\left(E_{k}(\lambda) - 1\right) + 1
	\end{align}
	for $k \geq 2$, and independently for $k = 1$ when
	\begin{align}
		E_2(\lambda) &= (1 + \lambda)\cdot \lambda + 1 \>, \nonumber \\
		\operatorname{deg} E_{k+1}(\lambda) &= 2\operatorname{deg}E_{k}(\lambda) \>.
	\end{align}
	If $\operatorname{deg}E_{k}(\lambda) = 2^{k-1}$, $\operatorname{deg}E_{k+1}(\lambda) = 2^{k + 1 - 1}$. This establishes the inductive step.
\end{proof}
\begin{proposition}
	If $E_{k}(\lambda) = \sum_{j = 0}^{2^{k-1}}E_{j, k}\lambda^{k}$, then all $E_{j, k}$ are positive integers,
	\begin{equation}
		E_{0, k} = E_{2^{k-1}, k} = 1 \>,
	\end{equation}
	and
	\begin{equation}
		e_k = E_{k}(1) = \sum_{j = 0}^{2^{k-1}}E_{j, k} \>.
	\end{equation}
\end{proposition}
\begin{proof}
	\begin{align}
		E_{k+1}(\lambda) &= E_{k}(\lambda)\left(E_{k}(\lambda) - 1\right) + 1 \nonumber \\
		&= \lambda E_{k}(\lambda)E_{k-1}(\lambda) \cdots E_{1}(\lambda) + 1
	\end{align}
	has trailing coefficient 1 (set $\lambda = 0$) and leading coefficient 1 (the square of the leading coefficient of $E_{k}(\lambda)$). As for $E_{j, k} \geq 1$ being integral, the Cauchy product formula gives
	\begin{equation}
		\left[ z^j \right] E_{k+1}(\lambda) = E_{j, k+1}
	\end{equation}
	(the coefficient of $z^j$ of $E_{k+1}$)
	\begin{equation}
		= \sum_{\ell = 0}^{j} E_{\ell, k}\hat{E}_{j - \ell, k}
	\end{equation}
	where
	\begin{equation}
		\hat{E}_{j - \ell, k} = 
		\begin{cases}
			E_{j - \ell, k} & \text{if } \ell < j \\
			0 & \text{if } \ell = j
		\end{cases}
	\end{equation}
	is a sum of products of positive integers, and hence a positive integer. The statement $e_{k} = \sum_{j = 0}^{2^{k-1}}E_{j, k}$ follows from the definition of $E_{j, k}$.
\end{proof}
\begin{proposition}
	\begin{equation}
		\max_{0 \leq j \leq 2^{k}} E_{j, k} \geq \left(\max_{0 \leq j \leq 2^{k-1}} E_{j, k}\right)^{2}.
	\end{equation}
\end{proposition}
\begin{proof}
	From the Cauchy product in the last proposition, if $j^*$ is the index of the largest coefficient of $E_k(\lambda)$, then for $j = 2j^*$ in $E_{k+1}(\lambda)$ the coefficient of $\left[ z^j \right]$ is
	\begin{equation}
		\sum_{\ell = 0}^{2j} E_{\ell, k}E_{2j - \ell, k}
	\end{equation}
	which, for $\ell = j^*$, contains
	\begin{equation}
		E_{j^*, k}E_{j^*, k} = E_{j^*, k}^{2}
	\end{equation}
	which establishes the proposition.
\end{proof}
\begin{proposition}
	The largest coefficient of $E_k(\lambda)$ grows doubly exponentially with~$k$.
\end{proposition}
\begin{proof}[Proof 1]
	\begin{equation}
		e_k = E_k(1) = \sum_{j = 0}^{2^{k-1}}E_{j, k} = \left\lfloor E^{2^{k}} + \dfrac{1}{2}\right\rfloor \>,
	\end{equation}
	then
	\begin{align}
		\max_{j}E_{j, k} &\geq \dfrac{1}{2^{k-1} + 1} \left \lfloor E^{2^{k}} + \dfrac{1}{2}\right \rfloor \nonumber\\
		&= E^{2^{k} - \mathcal{O}(k)} \>.
	\end{align}
\end{proof}
\begin{proof}[Proof 2]
	By inspection, $\max_{j} E_{j, 3} = 2$. Since $\max_{j} E_{j, 4} = 10 > 2^{2} = 2^{\sfrac{1}{4}\cdot 2^{3}}= 2^{\sfrac{1}{4}\cdot k}$, we are well on our way. Assume that $\max_{j} E_{j, k} = 2^{c_{i}2^{k}}$. Then $\max_{j} E_{j, k} \geq \left(2^{c_{1}\cdot 2^{k}}\right)^{2} = 2^{c_{1}2^{k+1}}$.
\end{proof}

\begin{proposition}
	The polynomials $E_{k}(\lambda)$ are all mutually relatively prime, as polynomials.
\end{proposition}
\begin{proof}
	The proof is the same as that proving the $e_{k}$ are relatively prime integers: $E_{n}(\lambda) \equiv 1\mod E_m(\lambda)$ if $n>m \Rightarrow \operatorname{gcd}(E_{n}(\lambda), E_{m}(\lambda)) = \operatorname{gcd}(1, E_{m}(\lambda)) = 1$.
\end{proof}

\begin{proposition}
	The roots of $E_{k}(\lambda)$ are simple.
\end{proposition}
\begin{proof}
	This is true for $E_{1}(\lambda)$ and $E_{2}(\lambda)$.
	
	Assume to the contrary that for some $k$ there exists a $\lambda^{*}$ for which both
	\begin{equation}
		E_{k+1}(\lambda^{*}) = 0
	\end{equation}
	and
	\begin{equation}
		E'_{k+1}(\lambda^{*}) = 0 \>.
	\end{equation}
	Then since
	\begin{equation}
		E_{j+1}(\lambda) = E_{j}(\lambda)\left(E_{j}(\lambda) - 1\right) + 1 \>, 
	\end{equation}
	we have
	\begin{equation}
		E'_{j+1}(\lambda) = \left(2E_{j}(\lambda) - 1\right) E'_{j}(\lambda) \>.
	\end{equation}
	Therefore, either $E_{k}(\lambda^{*}) = \sfrac{1}{2}$ (which is impossible because then $E_{k+1}(\lambda^{*}) = \sfrac{1}{2}(-\sfrac{1}{2}) + 1 = \sfrac{3}{4}\neq 0$) or $E'_{k}(\lambda^*) = 0$. If there exists any $\ell < k$ for which $E'_{\ell}(\lambda^*) \neq 0$ while $E'_{\ell + 1}(\lambda^*) = 0$, then $E_{\ell}(\lambda^*) = \sfrac{1}{2}$ because $E'_{\ell + 1}(\lambda) = \left(2E_{\ell}(\lambda) - \sfrac{1}{2}\right)E'_{\ell}(\lambda)$. If $E_{\ell}(\lambda^{*}) = \sfrac{1}{2}$, then $E_{j}(\lambda^*)$ for $j \geq \ell$ is rational because
	\begin{equation}
		E_{j+1}(\lambda^*) = E_{j}(\lambda^*)(E_{j}(\lambda^*)-1)
	\end{equation}
	is a product of rational numbers.
	
	This gives an ultimate contradiction because
	\begin{equation}
		E_{k}(\lambda^*)(E_{k}(\lambda^*) - 1) + 1 = 0
	\end{equation}
	only if $E_{k}(\lambda^*) = -\sfrac{1}{2} \pm \sfrac{i\sqrt{3}}{2} \notin \mathbb{Q}$. \mbox{} \\
\end{proof}
\begin{proposition}
	\begin{equation}
		\dfrac{1}{\lambda} = \sum_{k=1}^{n} \dfrac{1}{E_{k}(\lambda)} + \dfrac{1}{E_{n+1}(\lambda) - 1} \>.
		\label{eqn:egypt_poly}
	\end{equation}
\end{proposition}
\begin{proof}
	Identical to Lemma \ref{lemma:egypt} on substituting $E_{k}(\lambda)$ for $e_k$ and noting
	\begin{align}
		\dfrac{1}{\lambda} &= \dfrac{1}{\lambda + 1} + \dfrac{1}{\lambda^2 + \lambda} \nonumber \\
		&= \dfrac{\lambda}{\lambda^2 + \lambda} + \dfrac{1}{\lambda^2 + \lambda} \nonumber \\
		&= \dfrac{\lambda + 1}{\lambda(\lambda + 1)} \nonumber \\
		&= \dfrac{1}{\lambda} \>.
	\end{align}
\end{proof}
\begin{remark}
	The series in equation \eqref{eqn:egypt_poly} converges if $\lambda > 0$ and diverges if $\lambda = -\sfrac{1}{2}$.
\end{remark}
\begin{conjecture}
	There is convergence outside the ``cauliflower" in Figure \ref{fig:euclid} and divergence inside the cauliflower.
\end{conjecture}
\begin{definition}
	We say that a polynomial $p(\lambda)$ is \textsl{unimodal}~\cite{kalugin2010unimodal} if its coefficient vector $\left[ a_0, a_1, \cdots, a_n\right]$ of positive integers has first monotonic increase to a peak (which may occur twice or more at adjacent coefficients) and then decay to $a_n = 1$. Notice that $E_1(\lambda)$, $E_{2}(\lambda)$, $E_{3}(\lambda)$ and $E_{4}(\lambda)$ are unimodal.
\end{definition}
\begin{conjecture}
	The Euclid polynomials are unimodal.
\end{conjecture}
\begin{remark}
	The doubly exponential growth of the polynomial coefficients mean that the \textsl{conditioning} of the polynomial grows doubly exponentially in $k$. Note that since the degree $\operatorname{deg}E_{k} = 2^{k-1}$, this means that the conditioning grows exponentially in the degree. In contrast, we will see in section \ref{sec:condition_number} a much better condition number, sublinear in the degree. This means that evaluation (and rootfinding) requires significantly more precision (and therefore expense) if the monomial basis is used. The following definition is used in \cite{farouki1987numerical} and \cite{corless2013graduate}:
	\begin{equation}
		B_{k}(\lambda) = \sum_{j = 0}^{2^{k-1}} E_{j, k}\left| \lambda \right|^{j}
	\end{equation}
	as a ``condition number" for a given $\lambda$. One can show that if
	\begin{equation}
		p_{k}(\lambda) = \sum_{j = 0}^{2^{k-1}} E_{j, k}(1 + \delta_j)\lambda^{j}
	\end{equation}
	then $p_{k}(\lambda)$ differs from $E_k(\lambda)$ by at most 
	\begin{equation}
		\left| p_k(\lambda) - E_{k}(\lambda)\right| \leq B_k(\lambda) \cdot \max_{0 \leq j \leq 2^{k-1}} \left|\delta_{k}\right| \>.
	\end{equation}
	This shows that relative errors $\delta_k$ in the coefficients produce absolute errors in the values at most $B(\lambda)||\delta||_\infty$. From the foregoing discussion it is evident that on $0 \leq \lambda \leq 1$
	\begin{align}
		B_{k}(\lambda) &= \mathcal{O}\left(E^{2^{k}}\right) \\
		&= \mathcal{O}\left(e^{2\operatorname{deg}E_{k}(\lambda)}\right)
	\end{align}
	is exponentially large in the degree of $E_{k}(\lambda)$. That is, in order to ensure that numerical errors in evaluation (which, by standard backward error results are equivalent to $\mathcal{O}(\mu)$, where $\mu$ is the unit roundoff, relative changes in the coefficients) would require that the unit roundoff to be of size
	\begin{equation}
		\mu = \mathcal{O}\left(E^{-2\operatorname{deg}E_{k}(\lambda)}\right)
	\end{equation}
	which in turn requires $\mathcal{O}\left(2\operatorname{deg}E_{k}\right)$ bits of precision; this is an exponential number of bits of precision, in $k$. To evaluate $E_k(\lambda)$ (or to find its roots) one would need to use $\mathcal{O}\left(2^{k}\right)$ bit arithmetic. This is of course possible, but the cost of multiplication of high precision number grows faster than the precision length.	
\end{remark}
Luckily, there's a better way: minimal height companion matrices.

\section{A Brief History of the Technique}
In 2011, Piers W.~Lawrence invented a family of companion matrices for the Mandelbrot polynomials\footnote{It can be shown that the Euclid polynomials are related to the Mandelbrot polynomials. We can rewrite the Euclid polynomials as
\begin{align}
	f_{n+1} &= f_n^2 + \frac{1}{4} \nonumber \\
	4f_{n+1} &= \frac{1}{4}\left(4f_n\right)^2 + 1 \>. 
\end{align}
We can then let $u_n = 4f_n$, so
\begin{equation}
	u_{n+1} = \frac{1}{4}u_n^2 + 1 \>,
\end{equation}
which recurrence is the same as for the Mandelbrot polynomials, except with $z = \sfrac{1}{4}$ and
\begin{equation}
	u_1 = 4f_1 = 4\left(e_1 - \sfrac{1}{2}\right) = 2 \>;
\end{equation}
whereas $p_1 = 1$.
}, defined by $p_1(\lambda) = 1$ and for $n \geq 0$
\begin{equation}
	p_{n+1}(\lambda) = \lambda p_{n}^2(\lambda) + 1 \>.
\end{equation}
We have $p_2(\lambda) = \lambda + 1$ with a (trivial) companion matrix $\mathbf{M}_2 = \left[ -1 \right]$. Piers invented a recursive construction,
\begin{equation}
	\mathbf{M}_{n+1}\left[
		\begin{array}{ccc}
			\mathbf{M}_n & & -\mathbf{c}_n\mathbf{r}_n \\
			-\mathbf{r}_n & 0 & \\
			& -\mathbf{c}_n & \mathbf{M}_n
		\end{array}
	\right]
\end{equation}
where $\mathbf{r}_n = \left[\begin{array}{cccc} 0 & 0 & \cdots & 1\end{array}\right]$ and $\mathbf{c}_n = \left[\begin{array}{cccc} 1 & 0 & \cdots & 0 \end{array}\right]^{T}$, given
\begin{align}
	p_{n+1}(\lambda) &= \operatorname{det}\left(\lambda\mathbf{I} - \mathbf{M}_{n+1}\right) \nonumber \\ 	&= \lambda\operatorname{det}\left(\lambda\mathbf{I} - \mathbf{M}_n\right)^{2} + 1\>.
\end{align}

In her Masters' thesis~\cite{Chan2016}, Eunice Chan extended this construction to Fibonacci-Mandelbrot polynomials $q_n(\lambda)$ satisfying
\begin{align}
	q_0(\lambda) &= 0 \nonumber \\
	q_1(\lambda) &= 1 \nonumber \\
	q_{n+1}(\lambda) &= \lambda q_{n}(\lambda)q_{n-1}(\lambda) + 1
\end{align}
and Narayana-Mandelbrot polynomials $r_{n}(\lambda)$ satisfying
\begin{align}
	r_0(\lambda) &= 1 \nonumber \\
	r_1(\lambda) &= 1 \nonumber \\
	r_2(\lambda) &= 1 \nonumber \\
	r_{n+1}(\lambda) &= \lambda r_{n}(\lambda)r_{n-2}(\lambda) + 1 \>.
\end{align}
Chan used these to explore the comparative efficiency of linearization (companion matrices) and homotopy methods (i.e.~following paths, also called continuation methods, from roots of $p_{n}(\lambda)$ to roots of $p_{n+1}(\lambda)$ and similarly for the others). [Spoiler alert: homotopy wins, hands down.]

These families of polynomials all have similarities and it is not really surprising that analogues of Piers Lawrence's construction work to make companion matrices.

Donald E.~Knuth suggested we look at Euclid numbers (polynomials). The fact that it worked immediately suggested that the construction was in fact general, which led to the papers \cite{chan2017new} and \cite{chan2017constructing}.

We return from that generality to the Euclid polynomials, which are interesting enough in themselves to deserve further attention. In the rest of this paper, we show how this general technique of construction applies to the Euclid polynomials, how far we can push it, and what we learn in the process.

\section{Computation of eigenvalues}
\begin{figure}
	\centering
	\includegraphics[width = 0.75\textwidth]{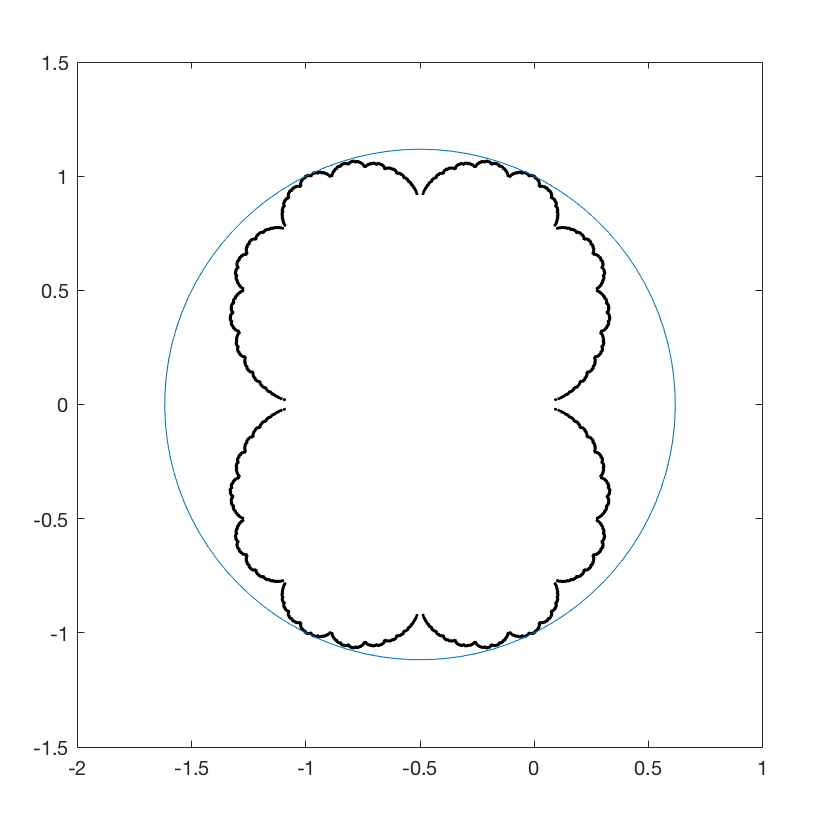}
	\caption{All $16,384$ roots of the Euclid polynomial $E_{15}(z)$ with circle of radius 1.1180, the approximate magnitude of the largest $\left| \lambda + \sfrac{1}{2}\right|$.}
	\label{fig:euclid}
\end{figure}

Suppose $E_k = \det\left(\lambda \mathbf{I} - \mathbb{E}_k\right)$. Each identity matrix $\mathbf{I}$ is a different size, but this should be natural enough: it will be $\operatorname{deg} E_k$ by $\operatorname{deg} E_k$ if it's being used in $\lambda \mathbf{I} - \mathbb{E}_k$. Notice that this amounts to a strong induction---we will need companion matrices for each prior polynomial in order to find one for $E_{k+1}$. Then put
\begin{equation}
	\widetilde{\mathbb{E}}_{k} :=
	\left[
	\begin{array}{c@{}c@{}c@{}c@{}c}
		\cdashline{1-1}
		\multicolumn{1}{:c:}{0} & & & & \mbox{\scriptsize{$\begin{array}{cc} &  \\ & \end{array}$}}\\
		\cdashline{1-2}
		\mbox{\scriptsize{$\begin{array}{cc}& -1 \\ {} & {} \end{array}$}} & \multicolumn{1}{:c:}{\mathbb{E}_1} & & & \\
		\cdashline{2-3}
		& \mbox{\scriptsize{$\begin{array}{cc}& -1 \\ {} & {} \end{array}$}} &  \multicolumn{1}{:c:}{\ddots} & & \\
		\cdashline{3-4}
		& & \mbox{\scriptsize{$\begin{array}{cc}& -1 \\ {} & {} \end{array}$}} &  \multicolumn{1}{:c:}{\mathbb{E}_{k-2}} & \\
		\cdashline{4-5}
		& & & \mbox{\scriptsize{$\begin{array}{cc}& -1 \\ {} & {} \end{array}$}} &  \multicolumn{1}{:c:}{\mathbb{E}_{k-1}}\\
		\cdashline{5-5}
	\end{array}
	\right]
	= \mathbb{E}_{k} - 
	\left[
		\begin{array}{cccc}
			0 & \cdots & 0 & 1 \\
			& & & 0 \\
			& & & \vdots \\
			& & & 0
		\end{array}
	\right] \>.
\end{equation}

\begin{remark}
	$\operatorname{det}\left(\lambda\mathbf{I} - \widetilde{\mathbb{E}}_{k} \right) = E_{k}(\lambda) - 1 = \lambda\sum_{j = 1}^{k-1}E_{j}(\lambda)$; subtracting $1$ just changes the final column of this companion (see \cite{chan2017constructing}).
\end{remark}

This is upper Hessenberg, but block lower triangular; therefore, its determinant is the product of the determinants of the blocks (see e.g.~\cite{horn2012matrix}) , and similarly for the resolvent~\cite{meyer2000matrix}, like so:
\begin{equation}
	\det \left( \lambda \mathbf{I} - \widetilde{\mathbb{E}}_{k}\right) = \lambda E_1(\lambda)E_2(\lambda) E_3(\lambda) \cdots E_{k-1}(\lambda) \>.
\end{equation}
Therefore, if we put a $1$ in the upper right corner (we will see shortly it must be $+1$),
\begin{equation}
	\mathbb{E}_{k+1} :=
	\left[
	\begin{array}{cc}
		\cdashline{1-1}
		\multicolumn{1}{:c:}{\widetilde{\mathbb{E}}_{k}} & \mbox{\scriptsize{$\begin{array}{cc} & 1 \\ & \end{array}$}} \\
		\cdashline{1-2}
		\mbox{\scriptsize{$\begin{array}{cc} & -1 \\ & \end{array}$}} & \multicolumn{1}{:c:}{\mathbb{E}_{k}}\\
		\cdashline{2-2}
	\end{array}
	\right] \>,
\end{equation}
we will have $E_{k+1}(\lambda) = \det \left( \lambda \mathbf{I} - \mathbb{E}_{k+1}\right)$ for $k \geq 2$ and $\mathbb{E}_{k+1}$ will be (irreducibly) upper Hessenberg if $\mathbb{E}_k$ is.

Explicitly, $\mathbb{E}_1 = \left[ -1 \right]$ and we may take
\begin{equation}
	\mathbb{E}_2 =
	\left[
	\begin{array}{cc}
		\cdashline{1-1}
		\multicolumn{1}{:c:}{\phantom{-}0} & \phantom{-}1 \\
		\cdashline{1-2}
		-1 & \multicolumn{1}{:c:}{-1}\\
		\cdashline{2-2}
	\end{array}
	\right]
\end{equation}
because $\det \left( \lambda \mathbf{I} - \mathbb{E}_2\right) = \det \left( \begin{array}{cc}\lambda & -1 \\ 1 & \lambda + 1 \end{array}\right) = \lambda\left(\lambda + 1 \right) + 1 = E_2(\lambda)$. Therefore,
\begin{equation}
	\mathbb{E}_{3} =
	\left[
		\begin{array}{cccc}
			\cdashline{1-1}
			\multicolumn{1}{:c:}{\phantom{-}0} & & & \phantom{-}1 \\
			\cdashline{1-2}
			-1 & \multicolumn{1}{:c:}{-1} & & \\
			\cdashline{2-4}
			& -1 & \multicolumn{1}{:c}{\phantom{-}0} & \multicolumn{1}{c:}{\phantom{-}1} \\
			& & \multicolumn{1}{:c}{-1} & \multicolumn{1}{c:}{-1} \\
			\cdashline{3-4}
		\end{array}
	\right] \>.
\end{equation}
To confirm, we form
\begin{equation}
	\lambda \mathbf{I} - \mathbb{E}_3 = \left[
	\begin{array}{cccc}
		\lambda & 0 & 0 & -1 \\
		1 & \lambda + 1 & 0 & \phantom{-}0 \\
		& 1 & \lambda & -1 \\
		& & 1 & \lambda + 1
	\end{array}
	\right] \>.
\end{equation}
A short computation shows
\begin{align}
	\det \left( \lambda \mathbf{I} - \mathbb{E}_3 \right) &= \lambda\left(\lambda + 1\right)\left(\lambda\left(\lambda + 1\right) + 1\right) + 1 \nonumber \\
	&= \lambda E_{1}(\lambda) E_{2}(\lambda) + 1 \nonumber \\
	&= E_{3}(\lambda)
\end{align}
as desired. Emboldened, we build
\begin{equation}
	\mathbb{E}_{4} = 
	\left[
		\begin{array}{cccccccc}
			\cdashline{1-1}
			\multicolumn{1}{:c:}{\phantom{-}0} & & & & & & & \phantom{-}1 \\
			\cdashline{1-2}
			-1 & \multicolumn{1}{:c:}{-1} & & & & & & \\
			\cdashline{2-4}
			& -1 & \multicolumn{1}{:c}{\phantom{-}0} & \multicolumn{1}{c:}{\phantom{-}1} & & & & \\
			& & \multicolumn{1}{:c}{-1} & \multicolumn{1}{c:}{-1} & & & & \\
			\cdashline{3-8}
			& & & -1 & \multicolumn{1}{:c}{\phantom{-}0} & & & \multicolumn{1}{c:}{\phantom{-}1} \\
			& & & & \multicolumn{1}{:c}{-1} & -1 & & \multicolumn{1}{c:}{} \\
			& & & & \multicolumn{1}{:c}{} & -1 & \phantom{-}0 & \multicolumn{1}{c:}{\phantom{-}1} \\
			& & & & \multicolumn{1}{:c}{} & & -1 & \multicolumn{1}{c:}{-1} \\
			\cdashline{5-8}
		\end{array}
	\right] 
\end{equation}
and direct computation again shows 
\begin{align}
	\operatorname{det}\left(\lambda\mathbf{I} - \mathbb{E}_{4}\right) &= \lambda\left(\lambda + 1\right)\left(\lambda\left(\lambda+1\right) + 1\right)\left(\lambda\left(\lambda + 1\right)\left(\lambda\left(\lambda + 1\right)+1\right) + 1\right) + 1 \nonumber \\
	&= \lambda E_{1}(\lambda) E_{2}(\lambda) E_{3}(\lambda) + 1 \nonumber \\
	&= E_{4}(\lambda) \>.
\end{align}

\begin{theorem}
	\begin{equation}
		E_{k}(\lambda) = \operatorname{det}\left(\lambda\mathbf{I} - \mathbb{E}_k\right)
	\end{equation}
	where $\mathbb{E}_{k}$ is defined as above.
\end{theorem}
\begin{proof}
	This follows immediately from Theorem 4 of \cite{chan2017constructing}. An easy proof follows from linearity of $\left(\lambda \mathbf{I} - \mathbb{E}_k\right)$ in its first row, and that the determinant of a block lower triangular matrix is the product of the determinants of the blocks; the 1 in the corner contributes $\left(-1\right)^{\operatorname{deg}(E_{k}(\lambda))-1}\cdot\left(-1\right)^{\operatorname{deg}(E_{k}(\lambda))-1} = +1$.
\end{proof}

\begin{lemma}
	The upper right corner of $\mathbb{E}_k$ is always $1$.
\end{lemma}

\begin{proof}
	As mentioned in Theorem 4 from \cite{chan2017constructing}, the element in the upper right corner is dependent on the degree of the polynomial, in this case $(-1)^{\operatorname{deg}E_k}$ for $\mathbb{E}_k$. Since the degree of the Euclid polynomials is
	\begin{align}
		\operatorname{deg} E_k &= 1 + \operatorname{deg} \left(E_{k-1}\right) - 1 + \operatorname{deg}E_{k-1} \nonumber \\
		&= 2\operatorname{deg}E_{k}
	\end{align}
	and $\operatorname{deg}E_1 = 1$; therefore,
	\begin{equation}
		\operatorname{deg}E_k = 2^{k-1} \>,
	\end{equation}
	which means that $\operatorname{deg}E_k$ is always even, and thus, the upper right corner of $\mathbb{E}_k$ is always $1$. We get $(-1)^{\operatorname{deg}e_n-1}$ from Laplace expansion and $(-1)^{\operatorname{deg}E_k-1}$ from minor and therefore,
	\begin{equation*}
		\left((-1)^{\operatorname{deg}E_k-1}\right)^2 = +1 \>.
	\end{equation*}
\end{proof}

\begin{remark}
	These ``Bohemian'' matrices\footnote{A matrix family is \textsl{Bohemian} if its entries come from a single discrete (and hence bounded) set. The name comes from ``\underline{Bo}unded \underline{He}ight \underline{M}atrix of Integers."} contain only entries that are $-1$, $0$, or $1$: the bound on that height of the entries is just $\left| m_{ij}\right| \leq 1$. But the coefficients of the Euclid polynomials $E_k(\lambda)$ are decidedly \textsl{not} bounded. This is just like the Mandelbrot polynomials, whose (polynomial coefficient) height grows exponentially with their degree $d_n = 2^{n-1} - 1$, and \textsl{doubly} exponentially with $n$. The eigenvalue problems we have found are considerably easier to solve than the monomial basis polynomials are!
\end{remark}

\begin{remark}
	There are many choices here---these companion matrices are in no way unique. For instance, we could use any of
	\begin{equation}
		\left[ 
		\begin{array}{cc}
			0 & \phantom{-}1 \\
			-1 & -1
		\end{array}
		\right], \quad
		\left[ 
		\begin{array}{cc}
			0 & -1 \\
			1 & -1
		\end{array}
		\right], \quad
		\left[ 
		\begin{array}{cc}
			-1 & -1 \\
			\phantom{-}1 & \phantom{-}0
		\end{array}
		\right], \quad
		\left[ 
		\begin{array}{cc}
			-1 & 1 \\
			-1 & 0
		\end{array}
		\right]
	\end{equation}
	for $\mathbb{E}_2$; and we may arrange the blocks for $\lambda$ (i.e. $[0]$), $E_1$, $E_2$, $\ldots$, $E_{k-1}$ in any order; at this time we do not know which order is best numerically, if any.
\end{remark}

\begin{figure}[t]
	\centering
	\includegraphics[width = 0.75\textwidth]{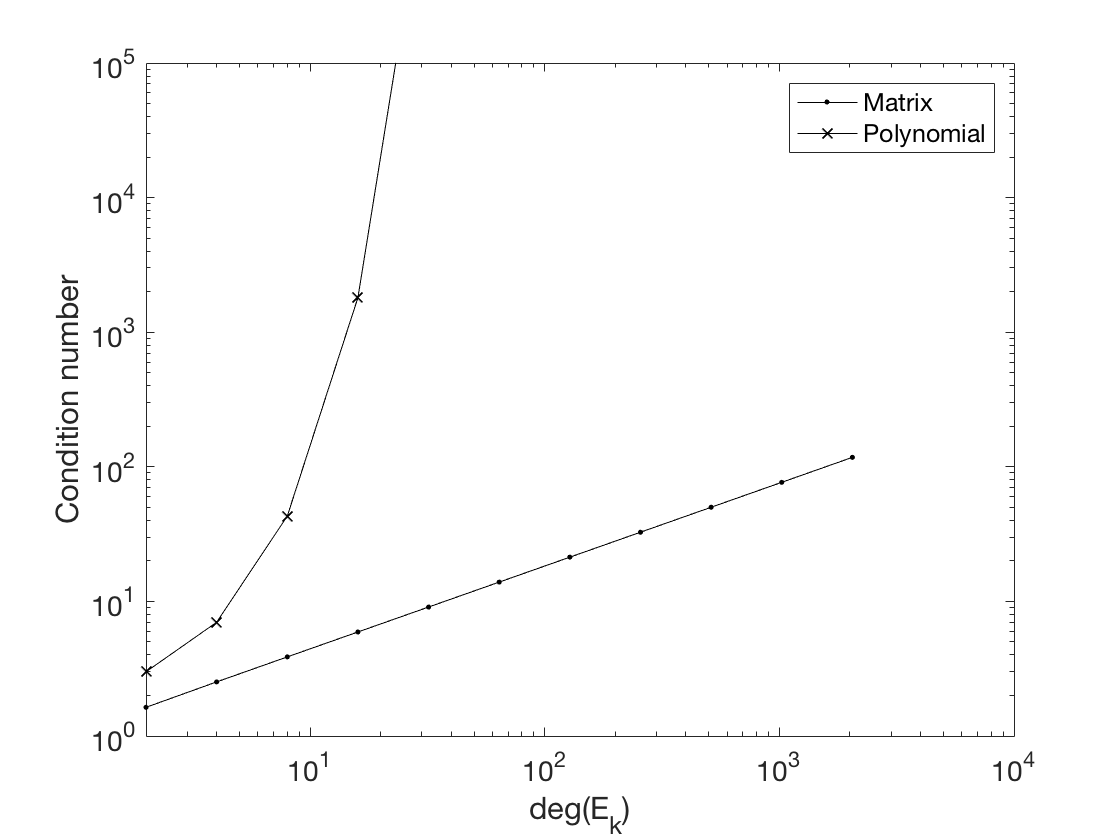}
	\caption{Log-log plot of condition numbers for the Euclid polynomials and their companions from $k = 2$ to $k = 12$. The computed slope for the condition number for the matrices is 0.618 giving an estimated condition number growth as $K_e\sim d^{0.618}$ which is better than the expected $\mathcal{O}(d^2)$ behaviour~\cite{beltran2017polynomial}. The curious three digit coincidence with $\sfrac{(\sqrt{5} - 1)}{2}$ is noted. The doubly exponential growth of the polynomial conditioning appears as exponential growth in this log log plot.}
	\label{fig:condition_number}
\end{figure}


\section{Conditioning of the eigenvalues of $\mathbb{E}_k$}\label{sec:condition_number}
Since the eigenvalues are all simple, $\mathbb{E}_{k}$ is diagonalizable and the condition number of each eigenvalue can be expressed using its unit left eigenvector $y^H$ and unit right eigenvector $x$ with $y^{H}\mathbb{E}_k = \lambda y^{H}$ and $\mathbb{E}_{k}x = \lambda x$, $\|x\| = \|y^{H}\| = 1$ and the condition number is 
\begin{equation}
	K_{e} = \sfrac{1}{(y^{H}x)}\>.
\end{equation} 
We expect from our experience with random matrices that $K_{e} = \mathcal{O}(d^{2})$ where $d$ is the dimension of the matrix, here the degree of the polynomial.

We can also look at the pseudospectra of the matrices that is, the eigenvalues of perturbed matrices~\cite{corless2013graduate}. Given an $\varepsilon > 0$, a pseudospectrum $\Lambda_{\varepsilon}(\mathbb{E}_{6})$ is defined by
\begin{equation}
	\Lambda_{\varepsilon}(\mathbb{E}_{6}) = \left\{ z \mathrel{\bigg|} \|(z\mathbf{I} - \mathbb{E}_{6})^{-1}\|_{2} \geq \dfrac{1}{\varepsilon}\right\} \iff \left\{ z \mathrel{\bigg|} \sigma_{\operatorname{deg}(E_{6})}\left(z\mathbf{I} - \mathbb{E}_{6}\right) \leq \varepsilon \right\} \>.
\end{equation}
Here $\sigma_{\operatorname{deg}_{E_{6}}}$ is the smallest singular value of $z\mathbf{I} - \mathbb{E}_{6}$. The contour plot can then be created using
\begin{equation}
	f(z) = \sigma_{\operatorname{deg}(E_{6})}\left(z\mathbf{I} - \mathbb{E}_{6}\right) > 0 \>.
\end{equation}
Figure \ref{subfig:pseudospectra} shows the pseudospectra of $\mathbb{E}_{6}$ for ten logarithmically-spaced values of $\varepsilon$ between $10^{-2}$ and $10^{-1}$.

To compare the conditioning of our companion matrices to the polynomials, we can also look at the pseudozeros of the polynomials. This allows us to look at the relationship between the condition number for the evaluation of polynomials and the condition number for rootfinding for polynomials~\cite{corless2013graduate}. The pseudozeros are defined as
\begin{equation}
	\Lambda_{\varepsilon}\left(E_{6}(\lambda)\right) = \left\{ \lambda \mathrel{\bigg|} \left| E_{6}(\lambda) \right| \leq \varepsilon \cdot B_{6}(\lambda)\right\} \>,
\end{equation}
where $B_{6}(\lambda) = E_{6}(\left| \lambda \right| )$. Figure \ref{subfig:pseudozeros} is a contour plot of $\sfrac{\left|E_{6}(\lambda)\right|}{E_{6}(\left| \lambda\right|)}$ between 10 logarithmically-spaced values between $10^{-5}$ and $10^{-4}$. 

\begin{figure}
	\parbox[b]{.5\linewidth}{
		\includegraphics[width=0.475\textwidth]{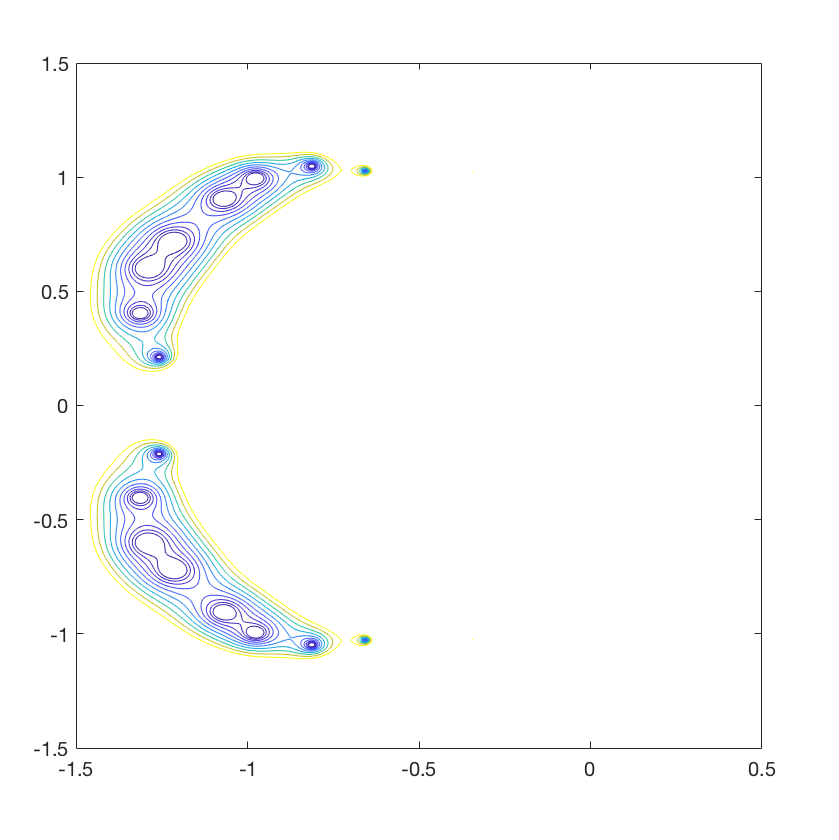}
	}%
	\hspace{.1\linewidth}%
	\parbox[b]{.4\linewidth}{%
		\subcaption{Pseudozeros of $E_{6}(\lambda)$ for 10 logarithmically-spaced values of $\varepsilon$ between $10^{-9.5}$ and $10^{-8.5}$. This is quite ill-conditioned. We only change $E_{6}(\lambda)$ by $3\times10^{-6}\%$ at most.\vspace{8em}}
		\label{subfig:pseudozeros}
	}
	\parbox[b]{.5\linewidth}{
		\includegraphics[width=0.475\textwidth]{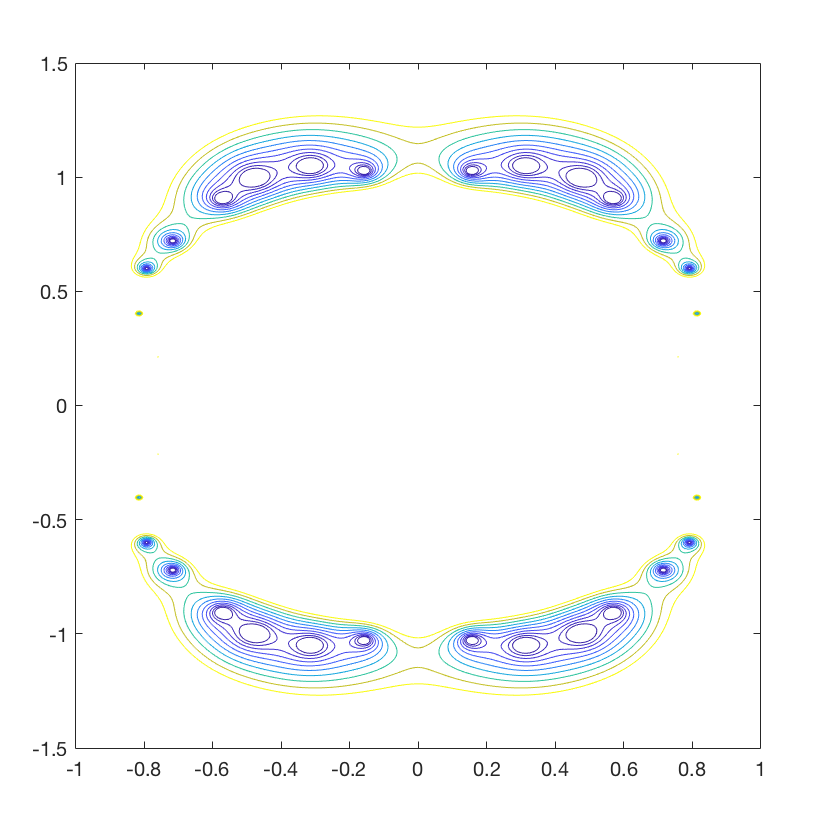}
	}%
	\hspace{.1\linewidth}%
	\parbox[b]{.4\linewidth}{%
		\subcaption{Pseudozeros of $E_{6}(u)$ for 10 logarithmically-spaced values of $\varepsilon$ between $10^{-3}$ and $10^{-2}$. This is substantially better-conditioned (and more symmetric) than the monomial basis (Figure \ref{subfig:pseudozeros}) changing $E_{k}(\lambda)$ by $1\%$ at most. \vspace{7em}}
		\label{subfig:pseudozeros_shift}
	}
	\parbox[b]{.5\linewidth}{
		\includegraphics[width=0.475\textwidth]{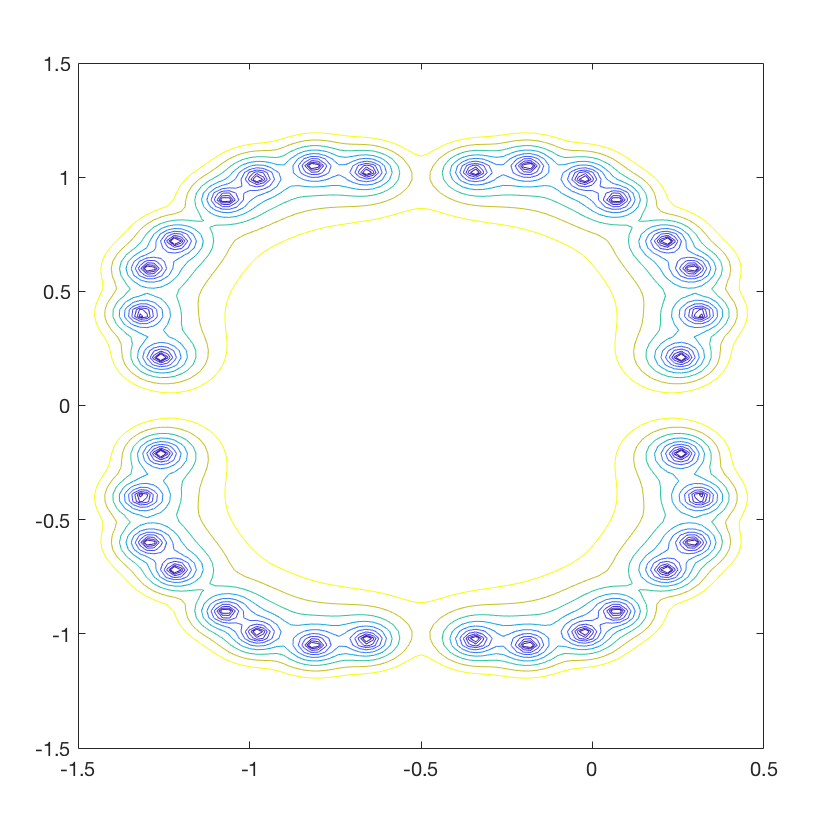}
	}%
	\hspace{.1\linewidth}%
	\parbox[b]{.4\linewidth}{%
		\subcaption{Pseudospectra of $\mathbb{E}_{6}$ for 10 logarithmically-spaced values of $\varepsilon$ between $10^{-2}$ and $10^{-1}$. This is the best-conditioned of the representations. This figure shows the results of changing $\mathbb{E}_{6}$ by 1--10\%.\vspace{8em}}
		\label{subfig:pseudospectra}
	}
\caption{The the similar spacings between Figures \ref{subfig:pseudozeros}, \ref{subfig:pseudozeros_shift} and \ref{subfig:pseudospectra} demonstrate the superior conditioning of the companion matrix, owing to its minimal height.}
\label{fig:pseudo}
\end{figure}

We can see from these figures that the roots computed from the companion matrix are well-conditioned. That the spacing are similar in the two figures, when $\varepsilon$ is so much smaller in Figure \ref{subfig:pseudozeros} demonstrates unequivocally that the eigenvalue problem is much better conditioned (a factor about $10^{3}$). This factor grows exponentially, as shown in Figure \ref{fig:condition_number}. We consider that these figures are ``similar" if
\begin{itemize}
	\item there are circles around individual roots/eigenvalues,
	\item there are some regions surrounding merged roots/eigenvalues,
	\item spacing between contours in about $1\%$ of the figure diameter.
\end{itemize}

\section{Do we have to use matrices?}
Expanding about $\lambda = -\sfrac{1}{2}$ is clearly better than expanding about $\lambda = 0$. Put $u = \lambda + \sfrac{1}{2}$, and then
\begin{align}
	E_{1}(\lambda) &= \lambda + 1 = u + \dfrac{1}{2} = E_{1}(u) \nonumber \\
	E_{2}(u) &= u^{2} + \dfrac{3}{4} \nonumber \\
	E_{3}(u) &= u^{4} + \dfrac{1}{2}u^{2} + \dfrac{13}{16} \nonumber \\
	E_{4}(u) &= u^{8} + u^{6} + \dfrac{7}{8}u^{4} + \dfrac{5}{16}u^{2} + \dfrac{217}{256}
\end{align}
and these polynomials only have even powers (after $k = 1$); this makes the polynomials subject to only half as much rounding error because zero coefficients cannot (are not allowed to) be perturbed. More, the coefficients of the even order terms appears to grow more slowly.

However, they do still grow doubly exponentially with $k$ (exponentially with the degree). The first polynomial to have a coefficient larger than 1 in magnitude is $E_{5}(u) = u^{16} + 2u^{14} + \cdots + \sfrac{57073}{65536}$ and thereafter the repeated squaring gives runaway growth. We present the graphs of the condition numbers
\begin{equation}
	\widetilde{B}_{k}(u) = \sum_{j=0}^{\operatorname{deg}E_{k}}\left|v_{j}\right|\left|u\right|^{k}
\end{equation}
on $0 \leq u \leq 1.1180$, a circle that contains the roots, in Figure \ref{fig:B_condition_number}. We see that for inside the interior of the cauliflower, this representation is well-conditioned (though uninteresting---nothing much is happening there) but near the boundary the exponential growth takes over.

We are forced to conclude that the minimal height companion matrices are exponentially better than these polynomials too.

Implicit in our discussion is the observation that the minimal height companion matrix is even more advantageous for larger $k$. The condition number of $E_{k}(\lambda)$ grows like $E^{2^{k}}$; the condition number of $E_{k}(u)$ grows like $E^{2^{k-1}}$ (possibly for a different $E$); while the condition number of $\mathbb{E}_{k}$'s eigenvalues grow only, as in Figure \ref{fig:condition_number}, like $\left(2^{k-1}\right)^{0.618}$. In practice, the pseudozeros/pseudospectra widths are already supporting this at $k = 6, 7, 8$, shown in Table \ref{tab:pseudo}.

\begin{table}[h]
	\centering
	\begin{tabular}{c|ccc}
		$k$ & $E_{k}(\lambda)$ & $E_{k}(u)$ & $\mathbb{E}_{k}$ \\
		\hline
		$6$ & $10^{-9.5} \ldots 10^{-8.5}$ & $10^{-3\phantom{0}} \ldots 10^{-2\phantom{0}}$ & $10^{-2} \ldots 10^{-1}$ \\
		$7$ & $10^{-19.5} \ldots 10^{-18.5}$ & $10^{-6\phantom{0}} \ldots 10^{-5\phantom{0}}$ & $10^{-2} \ldots 10^{-1}$ \\
		$8$ & $10^{-38.5} \ldots 10^{-37.5}$ & $10^{-12} \ldots 10^{-11}$ & $10^{-2} \ldots 10^{-1}$
	\end{tabular}
	\caption{Pseudozeros/pseudospectra of $E_{k}(\lambda)$, $E_{k}(u)$ and $\mathbb{E}_{k}$ for $k = 6, 7, 8$. For these $\varepsilon$ ranges, the pictures are similar to those of Figure \ref{fig:pseudo}. These pictures are available upon request.}
	\label{tab:pseudo}
\end{table}

\begin{remark}
	Using just the recurrence, not the polynomials, might be superior even to matrices.
\end{remark}

\begin{figure}[t]
	\centering
	\includegraphics[width=0.6\textwidth]{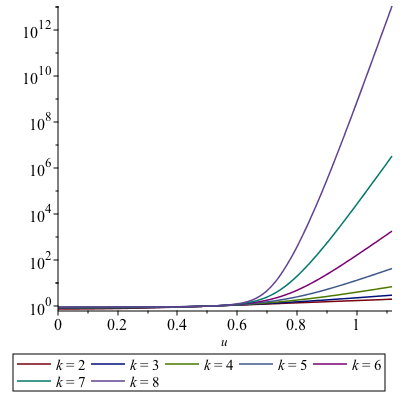}
	\caption{Condition numbers $\widetilde{B}_{k}(u)$ on $0 \leq u \leq 1.1180$ for $k = 2$ to $8$.}
	\label{fig:B_condition_number}
\end{figure}

\begin{figure}[t]
	\centering
	\includegraphics[width = 0.75\textwidth]{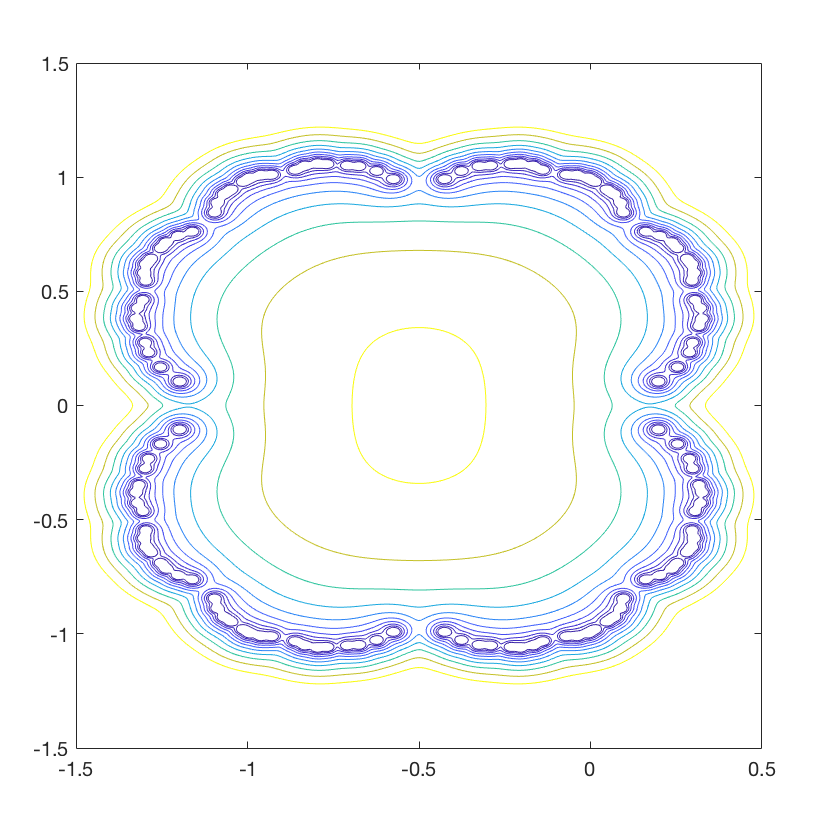}
	\caption{Pseudospectra of $\mathbb{E}_{8}$ for 10 logarithmically-spaced values of $\varepsilon$ between $10^{-2}$ and $10^{-1}$.}
	\label{fig:pseudospectra_8}
\end{figure}

\section{Concluding Remarks}
For us, the Euclid polynomials showed that the construction of companion matrices by the method of Piers Lawrence was, in fact, general. This construction also gives a \textsl{minimal height} companion matrix (over the integers); trivially so, because $\operatorname{height}(E_k) = 1$. This implies superior conditioning: already at $k = 6$, the matrix $\mathbb{E}_{6}$ has eigencondition about 1 while the polynomial $E_{6}(\lambda)$ had $B(\lambda) \sim 10^{4}$. But the other facts presented here show that the $E_k(\lambda)$ are themselves of interest: in particular, we're not done with the identity (for $\lambda > 0$)
\begin{equation}
	\dfrac{1}{\lambda} = \sum_{k \geq 1}\dfrac{1}{E_{k}(\lambda)} \>.
\end{equation}

\bibliographystyle{plain}
\bibliography{reference}

\subsection*{Acknowledgment}
We thank Donald E.~Knuth for his interest, suggestions, and improved proof. We also thank J.~Rafael Sendra, Juana Senda, and Laureano Gonzalez-Vega for their input and suggestions at the early stages of preparation of this paper. We also thank Dr.~Susan Colley for her feedback. This work was supported by the Natural Sciences and Engineering Research Council of Canada and an Ontario Graduate Scholarship.
\end{document}